\documentclass[11pt]{article}
\usepackage{amsmath, amssymb, amsthm, tocloft, enumerate, tikz,
            mathrsfs, color, comment, xspace, url} 
\usepackage[top=1in, bottom=1in, right=2cm, left=2cm]{geometry}
\usepackage{algorithmicx}
\usepackage{algorithm}
\usepackage{hyperref}
\usepackage{tikz}
\usetikzlibrary{arrows}
\usetikzlibrary{positioning, quotes}
\usetikzlibrary{shapes.geometric}
\usepackage[noend]{algpseudocode}
\algrenewcommand{\algorithmiccomment}[1]{\hfill[{\it #1}]}

\usepackage[capitalise]{cleveref}

\newcommand{\N}{\mathbb{N}}

\newcommand{\set}[2]{\{#1:#2\}}
\newcommand{\genset}[1]{\langle#1\rangle}
\newcommand{\id}{\operatorname{id}}

\renewcommand{\to}{\longrightarrow}

\newcommand{\s}{\mathbf{s}}
\renewcommand{\i}{\mathbf{i}}
\renewcommand{\t}{\mathbf{t}}
\newcommand{\lcm}{\operatorname{lcm}}

\newtheorem{thm}[equation]{Theorem}
\newtheorem{theorem}[equation]{Theorem}

\newtheorem{lem}[equation]{Lemma}
\newtheorem{cor}[equation]{Corollary}
\newtheorem{defn}[equation]{Definition}

\newtheorem{exam}[equation]{Example}
\newtheorem*{assumption}{Assumptions}

\newcommand{\InvPres}[2]{\mathrm{Inv}\langle #1 \mid #2 \rangle}

\DeclareMathOperator{\im}{im}
\DeclareMathOperator{\dom}{dom}

\usepackage[backend=biber, bibencoding=utf8, giveninits=true]{biblatex}
\begin{document}
\title{Counting monogenic monoids and inverse monoids}
\author{L. Elliott, A. Levine, and J. D. Mitchell\footnote{Email: jdm3@st-andrews.ac.uk}}
\maketitle
\begin{abstract} 
 In this short note, we show that the number of monogenic submonoids of the
  full transformation monoid of degree $n$ for $n > 0$, equals the sum of the
  number of cyclic subgroups of the symmetric groups on $1$ to $n$ points. 
  We also prove an analogous statement for monogenic subsemigroups of the
  finite full transformation monoids, as well as monogenic inverse submonoids and subsemigroups of the finite symmetric inverse monoids.
\end{abstract}

\section{Introduction}

In this short note we count, up to isomorphism, the number of monogenic subsemigroups and submonoids of the finite full transformation monoids and the number of
monogenic inverse subsemigroups and submonoids of the finite symmetric inverse monoids.

The question of counting the monogenic subsemigroups of the finite full transformation monoids 
was posed to the third author by A. Egri-Nagy in the course of their work on~\cite{East2017}. The general case of counting all of the subsemigroups of the finite full transformation monoids up to isomorphism seems extremely complex, but the highly restricted case of the monogenic subsemigroups is tractable. 
 The numbers of subsemigroups of the full transformation monoids and the number of inverse subsemigroups of the symmetric inverse monoids are very large (see for example Corollary 3.3 of \cite{Pyber93} or  Theorem 9.1 in \cite{Cameron2017}).
So, unsurprisingly, the monogenic semigroups account for a tiny proportion of the subsemigroups of these monoids.

The question of counting the monogenic inverse subsemigroups of the finite symmetric inverse monoids arose naturally in the context of another project of the authors of the present note. 
Every finite semigroup can be embedded in a full transformation monoid, and every finite inverse semigroup can be embedded in some symmetric inverse monoid.  
As such counting, or characterising, the transformation monoids, or inverse monoids, that can be found as subsemigroups of any given full transformation monoid, or symmetric inverse monoid, seems natural enough. 
The structure of monogenic semigroups is rather straightforward, and while the structure of monogenic inverse semigroups is more involved, they are sufficiently straightforward that it is possible to enumerate them. 
Similar questions have been studied for finite symmetric groups, see, for example, \cite{Holt2010} and the references therein, and for other classes of semigroups and monoids; see, for example,~\cite{Distler2012} and~\cite{Russell2021}.

There are a number of results in the literature relating to monogenic inverse semigroups. Preston~\cite{Preston1986aa} presented a description of all monogenic inverse monoids up to isomorphism. This description was independently obtained by Conway, Duncan and Paterson \cite{Conway1984} in the context of \(C^\ast\)-algebras. Dyadchenko \cite{Dyadchenko1984} gave a number of results on
monogenic inverse semigroups by studying free monogenic inverse semigroups. Some of the results in \cite{Dyadchenko1984} and  \cite{Preston1986aa} overlap with the results in the present paper, and the authors have endeavoured to point this out.

Throughout this short note, we denote the natural numbers $\{0, 1, \ldots\}$ by
$\N$.  If $n\in \N$ and $n \geq 0$, then the \textit{symmetric group}, denoted \(S_n\), is the group of all permutations
of the set \(\{1, \ldots, n\}\). The \textit{full transformation monoid} is the monoid of all functions from
the set \(\{1, \ldots, n\}\) to itself (called \textit{transformations}) and the semigroup operation is the usual composition of functions.
The \textit{symmetric inverse monoid} \(I_n\) is the set of all bijections between subsets of \(\{1, \ldots, n\}\), with the operation of composition of binary relations. Throughout the remainder of this note we will write functions to the right of their arguments and compose from left to right. 

We distinguish two notions of being ``generated'' by a single element, for monoids, and inverse monoids as follows.
A monoid \(M\) is \textit{monogenic} if there is \(m\in M\) such that \(M=\set{m^n}{n\in \N}\). An inverse monoid $M$ is \textit{a monogenic inverse monoid} if there exists $m\in M$ such that $M$ is the monoid generated by $m$ and $m ^ {-1}$.
Note that in a monogenic monoid $M$, or inverse monoid, the identity is $m ^ 0$ for every $m \in M$.
For finite groups, these two definitions coincide, and so a monogenic finite group is just a \textit{cyclic group}. We elaborate why we have chosen to consider monoids rather semigroups after stating the main theorem below.

Let $\mathbf{s},  \mathbf{t}, \mathbf{i} \colon \N \to \N$ be defined by 
\begin{align*}
  (n)\mathbf{s} & = \text{ the number of non-isomorphic cyclic subgroups of }
  S_n\\
  (n)\mathbf{t} & = \text{ the number of non-isomorphic monogenic submonoids of } T_n\\
  (n)\mathbf{i} & = \text{ the number of non-isomorphic monogenic inverse submonoids of } I_n.
\end{align*}
Since cyclic groups are determined up to isomorphism by their size, it follows
that $(n)\mathbf{s}$ is the number of distinct orders of elements in $S_n$; see~\cite{A030449} and \cref{table-1} for some values for $(n)\mathbf{s}$. Note that we will follow the convention that $(1)\s = 1$ (and $(0)\s = 1$ by virtue of the trivial group being cyclic).
A \textit{partition} of $n\in \N$ is a $k$-tuple
$(a_1, \ldots, a_k)$, where $k \geq 0$, $a_1\geq\cdots\geq a_k\geq1$ and
$a_1 + \cdots + a_k = n$. 
For instance, the partitions of $5$ are:
$$ 1 + 1 + 1 + 1 + 1, \quad 2 + 1 + 1 + 1, \quad 2 + 2 + 1,\quad 3 + 1 + 1,\quad 3 + 2,\quad 4 + 1,\quad
5.$$
There is a unique partition of $0$, namely, the empty partition $\varnothing$.
We use the standard conventions that an empty sum equals $0$ and an empty
product equals $1$; in particular, $\lcm(\varnothing)=1$.  

The order of a permutation $f\in S_n$ is just the least common multiple of the
lengths of its cycles, and so $(n)\s$ is the size of the set
$\set{\lcm(a_1, a_2, \ldots, a_k)}{a_1 + \cdots + a_k = n}$.

The purpose of this short note is to prove the following result.
\begin{thm}\label{thm-main}
  Let $n\in \N$ such that $n > 1$. Then \[
  (n)\mathbf{t}=\sum_{k = 1} ^ {n}(k)\mathbf{s}\quad \text{ and }\quad
 (n)\mathbf{i}= \sum_{k = 0} ^ {n } (k)\mathbf{s} 
                 = (n)\mathbf{t} + 1.\]
\end{thm}
See \cref{table-1} for the values of $(n)\mathbf{s}$, $(n)\mathbf{t}$, and $(n)\mathbf{i}$ for some small values of $n$. 

\begin{table}[h]
\begin{tabular}{r|r|r|r|r|r|r|r|r|r|r|r|r|r|r|r|r|r|r|r|r}
$n$ &0& 1 & 2 & 3 & 4 & 5 & 6 & 7 & 8 & 9 & 10 & 11 & 12 & 13 & 14 & 15 & 16 & 17 &
18 & 19 \\\hline
$(n)\mathbf{s}$&1 & 1& 2& 3& 4& 6& 6& 9& 11& 14& 16& 20& 23& 27& 31& 35& 43& 47& 55& 61  \\\hline
$(n)\mathbf{t}$ &1 & 1 & 3 & 6 & 10 & 16 & 22 & 31 & 42 & 56 & 72 & 92 & 115 & 142 & 173 & 208 & 251 & 298 & 353 & 414\\\hline
$(n)\mathbf{i}$ &1 & 2 & 4 & 7 & 11 & 17 & 23 & 32 & 43 & 57 & 73 & 93 & 116 & 143 & 174 & 209 & 252 & 299 & 354 & 415\\
\end{tabular}
\caption{The values of the functions $(n)\mathbf{s}$,  $(n)\mathbf{t}$, and $(n)\mathbf{i}$ for some small
values of $n$.}
\label{table-1}
\end{table}

A semigroup \(S\) is \textit{monogenic} if there is \(s\in S\) such that \(S=\set{s^n}{n>0, n\in \N}\).  As such a monogenic subsemigroup $S$ of $T_n$ contains the identity transformation $1_n\in T_n$; if and only if $S$ is a monogenic submonoid of $T_N$ (in this case $S$ is a cyclic group).

It follows that every monogenic submonoid $M$ of $T_n$ is either a a cyclic group; or $M \cap T_n\setminus S_n$ is a monogenic subsemigroup and $M \cap S_n = \{1_n\}$. Conversely, every monogenic subsemigroup of \(T_n\) is either a monogenic submonoid of \(T_n\) (and hence a group) or can be obtained from a monogenic submonoid by removing the identity element. Thus the number of monogenic subsemigroups of $T_n$ up to isomorphism is the number of cyclic subgroups of $S_n$ plus the number of monogenic subsemigroups of $T_n\setminus S_n$ that are not groups. 
As we will see later, the number of monogenic subsemigroups of  $T_n\setminus S_n$, that are groups is \((n - 1)\mathbf{s}\). Analogous comments hold for monogenic inverse subsemigroups and inverse submonoids of $I_n$. From this discussion we obtain the following corollary to \cref{thm-main}.

\begin{cor}
  The number of monogenic subsemigroups of \(T_n\) equals \((n)\mathbf{t} - (n - 1)\mathbf{s}\) up to isomorphism.
  The number of monogenic inverse subsemigroups of $I_n$, up to isomorphism, is \((n)\i - (n - 1)\mathbf{s}\).
\end{cor}


There is a natural injection from isomorphism types of submonoids of \(T_n\) to inverse submonoids of \(I_n\) defined by
mapping a generating transformation to a generating partial permutation with the same period and threshold. This is
almost a bijection, however it does not map onto any element of threshold \(n\) as \(T_n\) contains no such element. This
element is unique up to the isomorphism type of a submonoid it generates.

\section{Monogenic transformation monoids}

In this section, we collect a small number of facts about monogenic transformation monoids that we require to prove \cref{thm-main}.

If $M$ is a finite monoid and $x\in M$, then the
\textit{threshold} $t\in \N$ and \textit{period} $p\in \N$ of $x$ are the
least values such that $p > 0$ and $x ^ {t + p} = x^{t}$. 

\begin{lem}\label{lem-monogenic}
  Let $M$ be a monoid and let $a, b\in M$. Then the monogenic submonoids of $M$ generated by $a$ and $b$, respectively, are isomorphic if and only if the threshold and period of $a$ equal
  those of $b$. 
\end{lem}

To prove Lemma~\ref{lem-monogenic}, it is not difficult to show the unique homomorphism extending the map $a\mapsto b$ is an isomorphism whenever the thresholds and periods of $a$ and $b$ coincide.
A special case of Lemma~\ref{lem-monogenic} is when $M = S_n$, where the lemma
asserts that $\genset{a}$ and $\genset{b}$ are isomorphic if and only if
$|\genset{a}|= |\genset{b}|$, as mentioned above.

Recall that a \textit{digraph} $\Gamma$ is a pair $(V, E)$ consisting of a
\textit{vertex set} $V$ and an \textit{edge set} $E \subseteq V\times V$. A
digraph is \textit{functional} if for every $u\in V$ there exists a unique
$v\in V$ such that $(u, v) \in E$. Suppose that $D_n$ denotes the set of
functional digraphs with vertex set $\{1,\ldots,n\}$. It is
straightforward to verify that the function mapping $f$ to the functional
digraph $\Gamma_f$ with vertices $V = \{1, \ldots, n\}$ and edges $E=\set{(v,
(v)f)}{v\in V}$ is a bijection. We give an example of a functional digraph of
a transformation in Figure~\ref{fig:functional-digraph}.

\begin{figure}
    \centering
      \newcommand{\midarrow}{\tikz \draw[angle 90-] (0,0) -- +(.1,0);}
      \begin{tikzpicture}[scale=0.6]
        \node[] (1) at (-2.5, 0) {$1$};
        \node[] (2) at (2.5, 0) {$2$};
        \node[] (3) at (0, -4.75) {$3$};

        \node[] (4) at (-6, 0) {$4$};
        \node[] (7) at (-8.5, -2.5) {$7$};
        \node[] (8) at (-8.5, 2.5) {$8$};

        \node[] (5) at (6, 0) {$5$};
        \node[] (6) at (8.5, -2.5) {$6$};
        \node[] (9) at (8.5, 2.5) {$9$};
        
        \node[] (10) at (2.5, 3.5) {$10$};
        \node[] (11) at (2.5, 7) {$11$};
        
     \draw[-angle 90] (1) to [bend left] node[left] {} (2) ;
     \draw[-angle 90] (2) to [bend left] node[right] {} (3);
     \draw[-angle 90] (3) to [bend left] node[below] {} (1);
     \draw[-angle 90] (8) to (4);
     \draw[-angle 90] (7) to (4);
     \draw[-angle 90] (4) to (1);
     \draw[-angle 90] (11) to (10);
     \draw[-angle 90] (10) to (2);
     \draw[-angle 90] (5) to (2);
     \draw[-angle 90] (9) to (5);
     \draw[-angle 90] (6) to (5);
          
      \end{tikzpicture}
    \caption{Functional digraph of a transformation on \(11\) points with threshold \(2\)
    and period \(3\).}
    \label{fig:functional-digraph}
\end{figure}
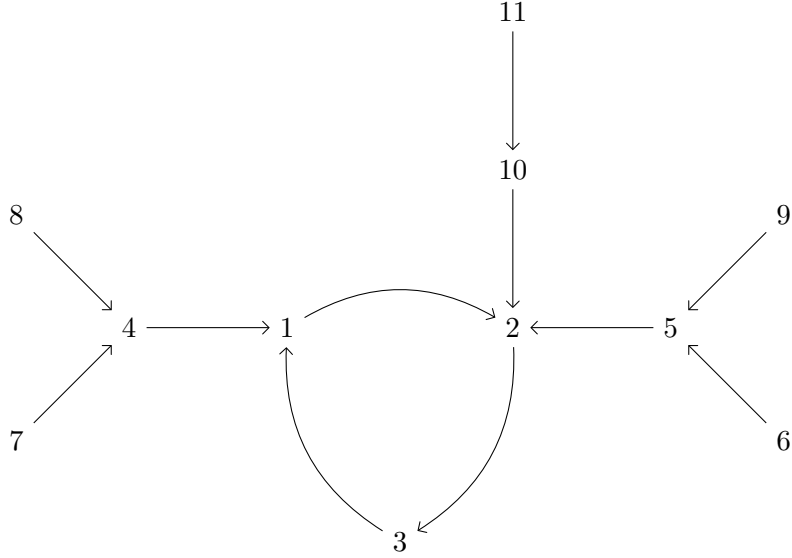

\begin{lem}
  Let \(n, p, t \in \mathbb{N}\) be such that $n> 0$ and $p> 0$. Then \(t\)
  and \(p\) are the threshold and period of an element of \(T_n\) if and only if there exists \(m \in \mathbb{N} \setminus \{0\}\) with $m\leq n$,
  such that \(p\) is the order of an element of \(S_m\), and \(t \in \{0, \ldots, n - m\}\).
\end{lem}

\begin{proof}
  The period of a transformation $f\in T_n$ is equal to the least common
  multiple $p$ of the lengths of the cycles in the digraph $\Gamma_f$, and the threshold
  of $f$ is equal to the maximal number $t$ of edges in a path $(x_0, x_1, \ldots, x_{t})$ where the vertex $x_{t}$ belongs to a cycle but the vertices $\{x_0, \ldots, x_{t - 1}\}$ do not.  In particular, the threshold of $f\in T_n$ is between $0$ and $n - 1$, inclusive. If $g \in S_m$ is any permutation, then there exists a
  transformation $f\in T_n$ such that the period of $f$ equals the order \(p\) of $g$
  and the threshold of $f$ is any value in $t \in \{0, \ldots, n - m\}$;
  one such transformation is
\begin{equation*}
  (i)f = 
  \begin{cases}
    (i)g  & \text{if } 1 \leq i \leq m    \\
    i - 1 & \text{if } m + 1\leq i \leq m + t \\
    i     & \text{if } i > m + t.
  \end{cases}
\end{equation*}
This transformation has threshold \(t\) and period \(p\), since
\begin{equation*}
(i)f^{t + p} = 
\begin{cases}
(i) g^{t + p} = (i) g^t = (i)f^k & \text{ if } 1 \leq i \leq m \\
     ((i)f^{i-m})f^{p + t-(i-m)} = (m)f^{p + t -(i-m)} = (m)f^{t-(i-m)}=(i)f^t & \text{ if } m + 1 \leq
i \leq m + t  \\
 (i)f^{t + p} = i = (i)f^t & \text{ if } i > m + t  
\end{cases}
\end{equation*}
shows that the threshold and period are at most \(t\) and \(p\), respectively. As \(f\) restricts to
a permutation of order \(p\), the period of \(f\) can be no less than \(p\), and so the period of
\(f\) equals \(p\).

If \(x<t\) and \(t> 0\), then \((m+t)f^{x} \notin \{1, 2, \ldots, m\} \), \((m+t)f^{x+pt}\in \{1, 2, \ldots, m\} \). In particular, \((m + t)f^x \neq
(m + t)f^{x + p}\), and so the threshold of \(f\) does not equal \(x\), and therefore it must equal \(t\).
\end{proof}

\section{Monogenic inverse monoids of partial permutations}

We now consider monogenic inverse submonoids of
\(I_n\). Many of the results in this section have
analogous or equivalent results in \cite{Preston1986aa}.
We include proofs for completeness.

Throughout this section we consider the monogenic inverse monoid $S_{n, k}$ defined by the following inverse monoid presentation (i.e. $S_{n, k}$ is isomorphic to the quotient of the monogenic free inverse monoid by the least congruence containing the relations):
\begin{equation}\label{eq-presentation}
\mathcal{P}_{n, k} = \InvPres{x}{x^nx^{-n} =  x^{n+1}x^{-(n+1)},  x^nx^{-n}= x^nx^{-n}x^k }
\end{equation}
for an arbitrary but fixed $k, n \in \N$ with $k > 0$.

By \cite{Munn1974}, elements of the free inverse monoid are uniquely determined by the corresponding \textit{Munn tree}. 
The Munn tree of an element of the monogenic free inverse monoid
can be defined for any $a, b, c\in \N$ such that $a, c \leq b$ to be 
the chain of length $b$ where the initial state is the node at distance $a$ from the start of the chain, and the terminal node is distance $c$ from the start. Hence every element of a monogenic free inverse monoid is represented by a word of the form:
  \[x^{-a}x^bx^{-b}x^c\in \{x, x ^{ -1}\} ^ *\]
where \(a, b, c\in \N\) and \(a, c\leq b\) and $\{x, x ^{ -1}\} ^ *$ is the free monoid over the alphabet $\{x, x ^{ -1}\}$. We give an example of a Munn tree in Figure~\ref{fig:munn_tree}.

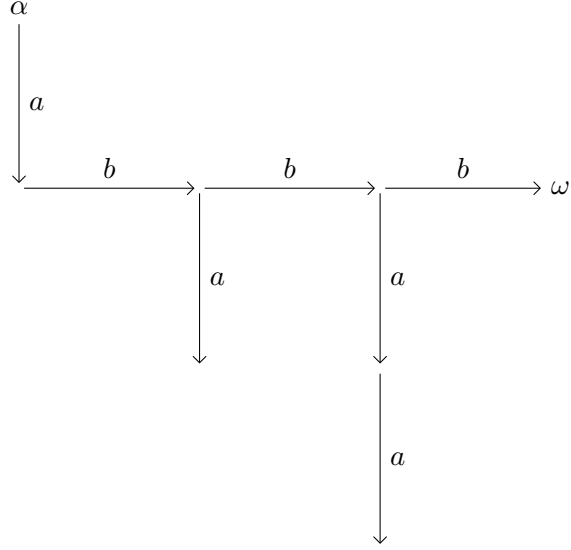
\begin{figure}
    \centering
          \begin{tikzpicture}
          [scale=0.8, rotate=270]
        \node (1) at (0, 0) {$\alpha$};
        \node[scale=0.5]  (a) at (3, 0) {};
        \node[scale=0.5]  (ab) at (3, 3) {};
        \node[scale=0.5]  (aba) at (6, 3) {};
        \node[scale=0.5]  (abb) at (3, 6) {};
        \node[scale=0.5]  (abba) at (6, 6) {};
        \node[scale=0.5]  (abbaa) at (9, 6) {};
        \node (abbb) at (3, 9) {\(\omega\)};

        \draw[-angle 90] (1) edge ["$a$"] (a);
        \draw[-angle 90] (a) edge ["$b$"] (ab);
        \draw[-angle 90] (ab) edge ["$a$"]  (aba);

        \draw[-angle 90] (ab) edge["$b$"] (abb);
        \draw[-angle 90] (abb) edge ["$a$"] (abba);
        \draw[-angle 90] (abba)edge ["$a$"] (abbaa);
        \draw[-angle 90] (abb) edge["$b$"] (abbb);

      \end{tikzpicture}
    \caption{Munn tree for the free inverse monoid element $abaa^{-1}baaa^{-1}a^{-1}b$ with start
    \(\alpha\) and end \(\omega\).}
    \label{fig:munn_tree}
\end{figure}

We start by defining a set of representative words for
each element, which we use to classify all finite
monogenic inverse monoids.

\begin{lem}[see also Proposition 4 of
\cite{Dyadchenko1984}]
   \label{inv_sgp_pres_lem}
   The set 
   \[ W:=\{x^{-a}x^bx^{-b}x^c \in \{x, x ^{ -1}\} ^ *: 0\leq a, c \leq b < n\}
   \cup \{x^{-a}x^nx^{-n}x^c \in \{x, x ^{ -1}\} ^ *: 0\leq a, c < k\}\]
   contains representatives for all elements of $S_{n, k}$. That is to say, if \(FIM(\{x\})\) is the free inverse monoid on \(\{x\}\), and \(\Psi: \{x , x ^ {-1}\} ^ {*} \to FIM(\{x\})\), \(\Phi: FIM(\{x\}) \to S_{n,k}\) are the natural homomorphisms, 
   then \((\Psi\Phi)\restriction_W\) is surjective.
\end{lem}
\begin{proof}
  As mentioned above, every element of  $S_{n, k}$, can be given as \(x^{-a}x^bx^{-b}x^c\) where \(a, c\leq b\). 
  If $b < n$, then \(x^{-a}x^bx^{-b}x^c\) belongs to the first set in the union in the statement of the lemma. 
  
  By the relations in \eqref{eq-presentation}, we know that 
  \(x^nx^{-n}= x^nx^{-n}x^k\) and 
  by inverting this it follows that
  \(x^nx^{-n}= x^{-k}x^nx^{-n}\).
  Hence, again using the given relations, if \(b\geq n\), then
  \[x^{-a}x^bx^{-b}x^c=x^{-a\text{ mod }k}x^{n}x^{-n}x^{c\text{ mod }k}.\]
  In particular, if $b \geq n$ and $a, c \leq b$, then $x^{-a}x^bx^{-b}x^c$ is equivalent to a word in the second set in the union in the statement of the lemma.
  %
\end{proof}

If $M$ is an inverse monoid, then 
we  denote the monogenic inverse submonoid of $M$ generated an element $f\in M$ by $\genset{f}$. 
An element $f \in I_n$ is called a \textit{chain} if 
the 
there exists $x\in \{1, \ldots, n\}\setminus \im(f)$
such that $\dom(f) = \set{(x)f ^ i}{i\in \{0, \ldots, |\dom (f)| -1\}}$ (and so $\im(f) = \set{(x)f ^ i}{i\in \{1, \ldots, n\}}$). 
The \textit{length} of a chain $f$ is denoted $|f|$ (the cardinality of $f$ as a subset of $\{1, \ldots, n\}\times \{1, \ldots, n\}$).
We denote a chain $f$ by $[x, (x)f, \ldots, (x)f ^ {|f| - 1}]$. Note that we will sometimes implicitly refer to a chain on \(1\) point of (length \(0\)), this should be interpreted as the empty function,
for example in \cref{lem-chain-cycle}, if \(a=1\) then \(x=(2, 3, \ldots, b+1)\).

To classify all monogenic inverse monoids, we must
first classify those generated by chains.

\begin{lem}
  \label{lemma-chain-gen}
  Let $f, g\in I_n$ be chains such that $|f| \geq |g|$. Then there exists a surjective homomorphism $\phi \colon \genset{f} \to \genset{g}$.
\end{lem}
\begin{proof}
Clearly if $|f| = |g|$, then the homomorphism from $\genset{f}$ to $\genset{g}$ mapping $f$ to $g$ is an isomorphism. Suppose that $|f| > |g|$.
We may assume without loss of generality that 
$f = [1, \ldots, n]$  and that $g = [1, \ldots, n-1]$. 
It suffices to show that there exists a set $X$ such that $\genset{g}$ acts faithfully by partial perms on $X$ and that $\genset{f}$ has the same action by partial perms on $X$ as $\genset{g}$. 
We set $X = \{(1, 2), \ldots, (n - 1, n)\}$ and define the actions of $f$ and $g$ by
\[
(a, a+1)f = ((a)f, (a+1)f) 
\quad (b, b+1)g = ((b)g, (b)g+1).
\]
It is routine to verify that these two actions are equal. 
\end{proof}

We require the following lemma, which is a special case of \cite[Theorem 7]{Preston1986aa}, we have included a proof here for the sake of completeness.

\begin{lem}\label{cor-chain-cycle}
        If $M$ is a finite monogenic inverse submonoid of \(I_n\), then there exist $a, b\in \N$ such that \(a+b=n\) and
    $M$ is isomorphic to the inverse submonoid of $I_n$ generated by 
    \[x=[1, \ldots, a]\cup p.\]
    where \(p\) is some permutation on the set \(\{a+1, \ldots, b+1\}\). 
    Moreover, the monoid is isomorphic to the submonoid of \(I_{a+|p|}\) generated by \([1, \ldots, a]\cup (a+1, \ldots, a+|p|)\).

\end{lem}
\begin{proof}
  If $M = \langle m \rangle$ is a monogenic inverse monoid, then $M$ is isomorphic to an inverse submonoid of $I_n$ for some $n$, and so we may suppose that $M$ is an inverse submonoid of $I_n$.  Thus $m$ is a union of disjoint cycles $y_1, \ldots, y_k$ and chains $h_1, \ldots, h_l$, and $M$ is an inverse submonoid of the direct product 
  \[
    M = \prod_{i=1}^k \langle y_i\rangle \times \prod_{j= 1}^l \langle h_j \rangle.
  \]
  If $k = 1$ and $l=1$, then there is nothing to prove, and so we suppose that $k \geq 2$ or $l \geq 2$.

  If $l \geq 2$, then we define $x$ to be a chain of length equal to the maximum of the lengths of $h_{m - 1}$ and $h_m$ 
  and we define 
  \[
    U = \prod_{i=1}^{k} \langle y_i\rangle \times \langle x\rangle \times \prod_{j = 1}^{l - 2} \langle h_j \rangle.
  \]
    We define $\phi\colon U \to M$ to be the homomorphism induced by the isomorphisms $y_i\mapsto y_i$ for all \(i\), $h_j\mapsto h_j$ for all \(j < l - 1\), and such that $x\mapsto (h_{l - 1}, h_l)$, which is also an isomorphism by~\cref{lemma-chain-gen}.
    Hence $\phi$ is an isomorphism when restricted to
    $\langle(y_1, \ldots, y_{k}, x, h_1, \ldots, h_{l - 2})\rangle$, and this monoid has a generator with one fewer chains than the generator of $M$.

     Suppose that $k\geq 2$. Then we define
$x$ to be any cycle of length $\lcm(|y_{k-1}|, |y_k|)$  and we define
  \[
    U = \prod_{i=1}^{k-2} \langle y_i\rangle \times \langle x\rangle \times \prod_{j= 1}^l \langle h_j \rangle.
  \]
  We define $\phi \colon U \to M$ to be the homomorphism induced by the isomorphisms $y_i\mapsto y_i$ for \(i < k - 1\), $h_j\mapsto h_j$ for all \(j\), and the isomorphism $\genset{x} \to \genset{(y_{k - 1}, y_k)}$ induced by $x\mapsto (y_{k - 1}, y_k)$. Then the restriction of $\phi$ to $\langle(y_1, \ldots, y_{k - 2}, x, h_1, \ldots, h_l)\rangle$ is an isomorphism of $\langle(y_1, \ldots, y_{k - 2}, x, h_1, \ldots, h_l)\rangle$ and $M$. In particular, these monoids are isomorphic, and the former has a generator with one fewer cycles than the generator of $M$.
\end{proof}

When one is not concerned with embedding an inverse monoid into a partial permutation monoid on a specific number of points, the following formulation of the above lemma is more natural.
\begin{lem}\label{lem-chain-cycle}
    If $M$ is a finite monogenic inverse monoid, then there exist $a, b, n\in \N$ such that 
    $M$ is isomorphic to the inverse submonoid of $I_n$ generated by 
    \[x=[1, \ldots, a]\cup (a+1, \ldots, a+b).\]
\end{lem}
\begin{proof}
    This is immediate from \cref{cor-chain-cycle}.
\end{proof}

We are now ready to give the theorem which allows us to classify the isomorphism types of finite monogenic inverse monoids. This theorem is essentially a reformulation of Theorem 7 from \cite{Preston1986aa}.
\begin{theorem}\label{thm-isomorphism}
  The inverse submonoid of \(I_{n+k}\) generated by a partial permutation
  \[[1, 2, \ldots , n]\cup (n+1, n+2, \ldots, n+k)\]
  is isomorphic to $S_{n, k}$ defined in \eqref{eq-presentation}
  for all \(n\geq 0\) and \(k\geq 1\).
  Moreover, if \(m\geq 0\) and \(l \geq 1\), then \(S_{n, k}\cong S_{m, l}\) if and only if \((n, k)=(m, l)\). 
\end{theorem}
\begin{proof}
  Since the set containing normal forms for $S_{n, k}$ from \cref{inv_sgp_pres_lem} is finite, the monoid $S_{n, k}$ is finite.
  It follows by \cref{lem-chain-cycle} that there are \(a, b\in \mathbb{N}\) such that $S_{n, k}$ is isomorphic to the inverse submonoid of \(I_{a+b}\) generated by the partial permutation
  \[x=[1, \ldots, a]\cup(a+1, \ldots, a+b).\]
  We will show that \(a=n\) and \(b=k\) (unless \(k=1\) in which case \(b\) may be \(0\)).
   The relation \( x^nx^{-n} =  x^{n+1}x^{-(n+1)}\) implies that
   \[[1, \ldots, a]^n[1, \ldots, a]^{-n}=[1, \ldots, a]^{n+1}[1, \ldots, a]^{-(n+1)}.\]
   This can only hold if \(a\leq n\).
   Thus, the relation  \(x^nx^{-n}= x^nx^{-n}x^k \)
   implies that 
   \[\id_{\{a + 1, \ldots, a + b\}} =x^nx^{-n}= x^nx^{-n}x^k=(a+1, \ldots, a+b)^k \]
   where \(\id_{\{a + 1, \ldots, a + b\}}\) is the identity function on the set \(\{a + 1, \ldots, a + b\}\).
   Thus \(b|k\) and, in particular \(b\leq k\) since \(k\geq 1\).
   
   It suffices to show that \(a\geq n\) and \(b\geq k \) (unless \(k=1\), in which case \(b\) is also allowed to be \(0\)).
   Seeking a contradiction suppose that \(a<n\). Since the relations in \eqref{eq-presentation} hold for $y = [1, \ldots, n]$, the inverse monoid generated by  $y$ is a homomorphic image of $S_{n, k}$. In particular, 
   $x^{a}x^{-a}\neq x^{n}x^{-n}$ since $y^{a}y^{-a}\neq y^{n}y^{-n}$ (the latter is the empty function and the former is not). 
   However, we proved above that
    \(x^{a}x^{-a}=\id_{\{a + 1, \ldots, a + b\}}=x^{n}x^{-n}\), which is a contradiction.
   
   If $1 < k$ and $b < k$, then a similar argument implies (the contradiction) that the relation \(x^{n}x^{-n}= x^{n}x^{-n}x^b\) does not hold in $S_{n, k}$ since it does not hold in the image of the homomorphism extending 
   \(x\mapsto (1, \ldots, k)\).
\end{proof}

\section{Proof of the main theorem}

\begin{proof}[Proof of Theorem \ref{thm-main}.]
We start by proving the formula for $(n)\mathbf{t}$.
Let $A$ be the set of pairs $(t, p)$ such that $t$ and $p$ are the
threshold and period of some transformation of degree $n$ and let 
$$A_t = \set{p\in\N\setminus\{0\}}{(t,p)\in A}.$$
Then, by Lemma~\ref{lem-monogenic}, $|A| = (n)\t$ and $|A| = |A_0| +
\cdots + |A_{n - 1}| = (n)\t$.  
It therefore suffices to show that
$|A_t|=\s(n-t)$ for all $t \in \{0, \ldots, n - 1\}$.

If $t \in \{0, \ldots, n - 1\}$ and $B_t$ is the set of orders of elements in
$S_{n-t}$, then clearly $|B_t| = (n - t)\s$ and it suffices to show that $A_t = B_t$.
If $p \in B_t$, then we showed above that there exists a transformation $f\in
T_n$ with threshold $t$ and period $p$ and so $p \in A_t$. 
Conversely, if $p\in A_t$, then, by the definition of $A_t$, there exists $f\in
T_n$ with threshold $t$ and period $p$.  Since the threshold of $f$ is $t$, there are
at most $n - t$ points in any cycle of $f$, and so $p$ is the order of a
permutation in $S_{n -t}$. In other words, $p\in B_t$. 
Thus
\[(n)\t = \sum_{i=0}^{n-1} |A_i| =\sum_{i= 0}^{n-1}(n-i)\s=\sum_{i= 1}^{n}(i)\s,\]
as required.

Next, we prove that the formula for $(n)\mathbf{i}$ in the statement of the theorem is correct. 
By \cref{cor-chain-cycle}, every monogenic inverse submonoid of \(I_n\) is isomorphic to one generated by the disjoint union of a chain and permutation. Moreover, by \cref{thm-isomorphism}, two such elements generate isomorphic inverse submonoids of \(I_n\) if and only if the chains have the same length and the permutations have the same order. 
Hence for every length of chain $i$, the number of distinct monogenic inverse submonoids up to isomorphism is $(n - i)\mathbf{s}$.
It follows that
\[(n)\mathbf{i}=  \sum_{i = 0}^n (n - i)\mathbf{s}= \sum_{i= 0}^{n }  (i)\mathbf{s}.\qedhere\]
\end{proof}

\section*{Acknowledgements}
The second and third authors would like to thank the London Mathematical Society, the Heilbronn Institute for Mathematical Research, and the University of St Andrews, for their support of this work. The authors would also like to thank  A. Egri-Nagy for originally asking the question that gave rise to this paper, and J. East for discussing the case of the full transformation monoid with the third author.

\printbibliography

\end{document}